\documentclass[a4paper,12pt]{article}
\usepackage{graphics}
\usepackage{latexsym}
\usepackage{times}
\usepackage{graphicx}
\catcode`@=11
\renewcommand{\@begintheorem}[2]{\it \trivlist            
      \item[\hskip \labelsep{\bf #1\ #2{\rm :}}]}         
\renewcommand{\@opargbegintheorem}[3]{\it \trivlist       
      \item[\hskip \labelsep{\bf #1\ #2\ {\rm (#3)\/:}}]}
\def\@sect#1#2#3#4#5#6[#7]#8{\ifnum #2>\c@secnumdepth
     \def\@svsec{}\else 
     \refstepcounter{#1}\edef\@svsec{\csname the#1\endcsname{.}\hskip 1em }\fi
     \@tempskipa #5\relax
      \ifdim \@tempskipa>\z@ 
        \begingroup #6\relax
          \@hangfrom{\hskip #3\relax\@svsec}{\interlinepenalty \@M #8\par}
        \endgroup
       \csname #1mark\endcsname{#7}\addcontentsline
         {toc}{#1}{\ifnum #2>\c@secnumdepth \else
                      \protect\numberline{\csname the#1\endcsname}\fi
                    #7}\else
        \def\@svsechd{#6\hskip #3\@svsec #8\csname #1mark\endcsname
                      {#7}\addcontentsline
                           {toc}{#1}{\ifnum #2>\c@secnumdepth \else
                             \protect\numberline{\csname the#1\endcsname}\fi
                       #7}}\fi
     \@xsect{#5}}

\@addtoreset{equation}{section}
\newcommand{\Delete}[1]{}

\usepackage{amsmath,amssymb,amsthm}

\theoremstyle{plain}
\newtheorem{Thm}{Theorem}[section]
\newtheorem{Lem}[Thm]{Lemma}
\newtheorem{Rem}[Thm]{Remark}

\newcommand{\cS}{\ensuremath{\mathcal{S}}}

\title{A generalization of Opsut's lower bounds \\
for the competition number of a graph}

\author{
\begin{tabular}{c}
{\sc Yoshio SANO}\\
\\
National Institute of Informatics \\
Tokyo 101-8430, Japan \\
{\tt sano@nii.ac.jp}
\end{tabular}
}

\date{}

\begin{document}

\maketitle

\begin{abstract}
The notion of a competition graph was introduced by J. E. Cohen in 1968. 
The {\it competition graph} $C(D)$ of a digraph $D$ 
is a (simple undirected) graph which 
has the same vertex set as $D$ 
and has an edge between two distinct vertices $x$ and $y$ 
if and only if 
there exists a vertex $v$ in $D$ such that 
$(x,v)$ and $(y,v)$ are arcs of $D$. 
For any graph $G$, $G$ together with sufficiently many isolated vertices 
is the competition graph of some acyclic digraph. 
In 1978, F. S. Roberts defined 
the {\it competition number} $k(G)$ of a graph $G$ 
as the minimum number of such isolated vertices. 
In general, it is hard to compute the competition number 
$k(G)$ for a graph $G$ and it has been one of the 
important research problems in the study of 
competition graphs 
to characterize a graph by its competition number. 
In 1982, R. J. Opsut gave two lower bounds 
for the competition number of a graph. 
In this paper, 
we give a generalization of these two lower bounds for 
the competition number of a graph. 

\end{abstract}

\vspace{4mm}


\vspace{4mm}

\newpage
\section{Introduction}

Throughout this paper, all graphs $G$ are finite, simple, and undirected. 
The notion of a competition graph was introduced by J. E. Cohen \cite{Cohen1} 
in connection with a problem in ecology. 
The {\it competition graph} $C(D)$ of a digraph $D$ is the graph
which has the same vertex set 
as $D$ and has an edge between two distinct vertices 
$x$ and $y$ if and only if 
there exists a vertex $v$ in $D$ 
such that $(x,v)$ and $(y,v)$ are arcs of $D$. 
For any graph $G$, $G$ together 
with sufficiently many 
isolated vertices is the competition graph 
of an acyclic digraph. 
From this observation, 
F. S. Roberts \cite{MR0504054} defined the {\it competition number} $k(G)$ of
a graph $G$ to be the minimum number $k$ such that $G$ together with 
$k$ isolated vertices is the competition graph of an acyclic digraph: 
\begin{equation}
k(G):= \min \{k \in \mathbb{Z}_{\geq 0} \mid G \cup I_k = C(D) 
\text{ for some acyclic digraph } D \}, 
\end{equation}
where $I_k$ denotes a set of $k$ isolated vertices. 

A digraph is said to be {\it acyclic} 
if it contains no directed cycles. 
For a digraph $D$, 
an ordering $v_1, v_2, \ldots, v_n$ of the vertices of $D$
is called an {\it acyclic ordering} of $D$
if $(v_i,v_j) \in A(D)$ implies $i<j$.
It is well-known that a digraph $D$ is acyclic if and only if
there exists an acyclic ordering of $D$.
For a vertex $v$ in a digraph $D$, 
the {\it out-neighborhood} of $v$ in $D$ 
is defined to be the set 
$\{w \in V(D) \mid (v,w) \in A(D) \}$ 
and is denoted by $N_D^+(v)$, and 
the {\it in-neighborhood} of $v$ in $D$ 
is defined to be the set 
$\{ u \in V(D) \mid (u,v) \in A(D) \}$ 
and is denoted by $N_D^-(v)$.

For a vertex $v$ in a graph $G$, 
the {\it (open) neighborhood} of $v$ in $G$ 
is defined to be the set 
$\{u \in V(G) \mid uv \in E(G) \}$ 
and is denoted by $N_G(v)$. 
A subset $S \subseteq V(G)$ of the vertex set of a graph $G$ 
is called a {\it clique} of $G$ if the subgraph $G[S]$ 
of $G$ induced by $S$ is a complete graph. 
For a clique $S$ of a graph $G$ and an edge $e$ of $G$,
we say {\it $e$ is covered by $S$} 
if both of the endpoints of $e$ are contained in $S$.
An {\it edge clique cover} of a graph $G$
is a family of cliques of $G$ such that
each edge of $G$ is covered by some clique in the family 
(see \cite{MR770871} for applications of edge clique covers). 
The {\it edge clique cover number} $\theta_E(G)$ of a graph $G$ 
is the minimum size of an edge clique cover of $G$. 
A {\it vertex clique cover} of a graph $G$ 
is a family of cliques of $G$ such that 
each vertex of $G$ is contained in some clique in the family. 
The {\it vertex clique cover number} $\theta_V(G)$ of a graph $G$
is the minimum size of a vertex clique cover of $G$.

R. D.  Dutton and R. C. Brigham \cite{MR712930} 
characterized the competition graphs of acyclic digraphs 
in terms of edge clique covers. 
(F. S. Roberts and J. E. Steif \cite{MR712932} 
characterized the competition graphs of loopless digraphs. 
J. R. Lundgren and J. S. Maybee \cite{MR712931} 
gave a characterization of graphs whose competition number 
is at most $m$.) 

However, R. J. Opsut \cite{MR679638} showed that 
the problem of determining 
whether a graph is the competition graph of an acyclic digraph 
or not is NP-complete. 
It follows that the computation of the 
competition number of a graph is an NP-hard problem, 
and thus 
it does not seem to be easy in general to compute 
$k(G)$ for an arbitrary graph $G$ 
(see \cite{kimsu} and \cite{kr} 
for graphs whose competition numbers are known). 
It has been one of the important research problems in the study of 
competition graphs to characterize a graph by its competition number. 

R. J. Opsut gave the following two lower bounds for the competition number 
of a graph. 

\begin{Thm}[Opsut {\cite[Proposition 5]{MR679638}}]\label{thm:OpsutBdE}
For any graph $G$, 
\begin{equation}\label{eq:OpsutBdE}
k(G) \geq \theta_E(G)-|V(G)|+2. 
\end{equation}
\end{Thm}

\begin{Thm}[Opsut {\cite[Proposition 7]{MR679638}}]\label{thm:OpsutBdV}
For any graph $G$, 
\begin{equation}\label{eq:OpsutBdV}
k(G) \geq \min \{ \theta_V(N_{G}(v)) \mid v \in V(G) \}. 
\end{equation}
\end{Thm}

\noindent
These seem to be the only known sharp lower bounds for 
an arbitrary graph $G$. 

In this paper, 
we give a generalization of 
these two lower bounds which contains both as special cases. 
In particular, our main result contains both lower bounds 
given in Theorems \ref{thm:OpsutBdE} and \ref{thm:OpsutBdV} 
as special cases. 
The proof of our main result is elementary, 
but the new lower bound given in this paper 
would be a strong tool in the study of the competition number of a graph.

\section{Main Result}

Let $G$ be a graph and $F \subseteq E(G)$ be a subset of the edge set of $G$. 
An {\it edge clique cover} of $F$ in $G$ is a family of cliques of $G$ 
such that each edge in $F$ is covered by some clique in the family. 
We define the {\it edge clique cover number} 
$\theta_E(F;G)$ of $F \subseteq E(G)$ in $G$ as 
the minimum size of an edge clique cover of $F$ in $G$: 
\begin{equation}
\theta_E(F;G) := \min \{|\cS| \mid \cS \text{ is an edge clique cover of } 
F \text{ in } G \}. 
\end{equation}
By definition, it follows that 
the edge clique cover number 
$\theta_E(E(G); G)$ 
of $E(G)$ in a graph $G$ is equal to 
the edge clique cover number $\theta_E(G)$ of the graph $G$. 

Let $G$ be a graph and $U \subseteq V(G)$ be a subset of the vertex set of $G$. We define 
\begin{eqnarray}
N_G[U] &:=& \{v \in V(G) \mid v 
\text{ is adjacent to a vertex in } U \} \cup U, \\
E_G[U] &:=& \{e \in E(G) \mid e \text{ has an endpoint in } U \}. 
\end{eqnarray}
We denote by the same symbol $N_G[U]$ the subgraph of $G$ induced by $N_G[U]$. 
Note that $E_G[U]$ is contained in the edge set of the subgraph $N_G[U]$. 
We denote by ${V \choose m}$ the set of all $m$-subsets of a set $V$. 

Now we are ready to state our main result. 

\begin{Thm}\label{thm:main}
Let $G=(V,E)$ be a graph. Then 
\begin{equation}\label{eq:thm:main}
k(G) \geq 
\max_{m \in \{1, \ldots, |V|\} } 
\min_{U \in {V \choose m}} 
\Big{(}
\theta_E(E_{G}[U]; N_{G}[U]) -|U|+1 
\Big{)}. 
\end{equation} 
\end{Thm}

To prove our main theorem, we show the following lemma. 

\begin{Lem}\label{lem:main}
Let $G=(V,E)$ be a graph. 
Let $m$ be an integer such that $1 \leq m \leq |V|$. 
Then 
\begin{equation}\label{eq:lem:main}
k(G) \geq 
\min_{U \in {V \choose m}}
\theta_E(E_{G}[U]; N_{G}[U]) -m+1. 
\end{equation}
\end{Lem}

\begin{proof}
Let $k:=k(G)$ for convenience. Fix an integer $m$ such that 
$1 \leq m \leq |V|$. 
Let $D$ be a 
minimal 
acyclic digraph 
with respect to the number of arcs 
such that $C(D)=G \cup I_k$, 
where $I_k:=\{z_1, \ldots, z_k\}$ is a set of $k$ isolated vertices. 
Let $v_1, \ldots, v_n, z_1, \ldots, z_k$ 
be an acyclic ordering of $D$, and 
let $W:=\{v_{n-m+1}, \ldots, v_{n} \}$. 
Note that $|W|=m$. 
Let 
\[
\cS:=\{N_D^-(w) \cap N_G[W] \mid w \in (W \cup I_k) 
\setminus \{v_{n-m+1}\} \}. 
\] 
For each $w \in (W \cup I_k) 
\setminus \{v_{n-m+1}\}$, 
since $N_D^-(w)$ forms a clique of the graph $G$, 
the set $N_D^-(w) \cap N_G[W]$ forms a clique of the induced subgraph 
$N_G[W]$ of $G$. Thus $\cS$ is a family of cliques of $N_G[W]$. 

Take any edge $e=uv \in E_G[W]$, where $u \in W$ and $v \in N_G(u)$. 
Since $u$ and $v$ are adjacent, there exists a common out-neighbor 
$w \in N_D^+(u) \cap N_D^+(v)$. 
Then $\{u,v\} \subseteq N^-_D(w)$. 
Since $v_1, \ldots, v_n, z_1, \ldots, z_k$ 
is an acyclic ordering of $D$, 
the out-neighborhood $N_D^+(u)$ of $u$ in $D$ 
is contained in the set 
$(W \cup I_k) \setminus \{v_{n-m+1}\}$ 
for each vertex $u \in W$. 
Therefore $N_D^+(u) \cap N_D^+(v) \subseteq 
(W \cup I_k) \setminus \{v_{n-m+1} \}$ 
and so $w \in (W \cup I_k) \setminus \{v_{n-m+1} \}$. 
Then it follows that 
the edge $e$ is covered by $N_D^-(w) \cap N_G[W] \in \cS$. 

Thus the family $\cS$ 
is an edge clique cover of $E_{G}[W]$ in $N_G[W]$. 
This implies that 
$\theta_E(E_G[W]; N_G[W]) \leq |\cS| = m+k-1$, 
that is, 
$\theta_E(E_G[W]; N_G[W]) -m+1 \leq k$. 
Hence 
\[
\min_{U \in {V \choose m}} \theta_E(E_G[U]; N_G[U]) -m+1 
\leq \theta_E(E_G[W]; N_G[W]) -m+1 \leq k(G), 
\]
and the lemma holds. 
\end{proof}

\begin{proof}[Proof of Theorem \ref{thm:main}]
Since the inequality (\ref{eq:lem:main}) holds for any 
$m \in \{1, \ldots, |V|\}$, 
it follows that the inequality (\ref{eq:thm:main}) holds. 
\end{proof}


\begin{Rem}
{\rm 
Consider the case $m=1$ in the inequality (\ref{eq:lem:main}). 
Then we obtain 
\[
k(G) \geq \min_{v \in V(G)} \theta_E(E_G[v]; N_G[v] ) . 
\]
Since a family $\{S_1, \ldots, S_r \}$ of cliques is an edge clique cover of 
$E_G[v]$ in $G$ if and only if 
$\{S_1 \cap N_G[v], \ldots, S_r \cap N_G[v] \}$ 
is an edge clique cover of $E_G[v]$ in $N_G[v]$, 
it holds that 
$\theta_E(E_G[v]; N_G[v]) = \theta_E(E_G[v];G)$. 
Since a family $\{S_1, \ldots, S_r \}$ of cliques is an edge clique cover of 
$E_G[v]$ in $G$ if and only if 
$\{S_1 \setminus \{v\}, \ldots, S_r \setminus \{v\} \}$ 
is a vertex clique cover of $N_G(v)$ in $G$, 
it holds that 
$\theta_E(E_G[v];G) = \theta_V(N_G(v))$. 
Therefore we have 
$\theta_E(E_G[v];N_G[v]) = \theta_V(N_G(v))$. 
Hence the above inequality coincides with 
the lower bound (\ref{eq:OpsutBdV}) 
in Theorem \ref{thm:OpsutBdV}. 
}
\end{Rem}

\begin{Rem}
{\rm 
Consider the case $m=|V|-1$ in the inequality (\ref{eq:lem:main}). 
Then we obtain 
\[
k(G) \geq \min_{v \in V} \theta_E(E_G[V \setminus \{v\}]; 
N_G[V \setminus \{v\}] ) - |V|+2. 
\]
Since $G=(V,E)$ has no loops, 
it holds that $E_G[V \setminus \{v\}] = E$. 
If the vertex $v$ is not isolated in $G$, 
then we have $N_G[V \setminus \{v\}] = V$ 
and thus 
$\theta_E(E_G[V \setminus \{v\}]; N_G[V \setminus \{v\}] ) = 
\theta_E(E; G) = \theta_E(G)$. 
If $v$ is an isolated vertex, 
then we have $N_G[V \setminus \{v\}]=V \setminus \{v\}$ 
and thus 
$\theta_E(E_G[V \setminus \{v\}]; N_G[V \setminus \{v\}] ) = 
\theta_E(E; G-\{v\}) = \theta_E(E; G) = \theta_E(G)$. 
Hence the above inequality coincides with 
the lower bound (\ref{eq:OpsutBdE}) 
in Theorem \ref{thm:OpsutBdE}. 
}
\end{Rem}

\begin{Rem}
{\rm 
The new lower bound given in Theorem \ref{thm:main} 
is a strong tool to compute the exact values of 
the competition numbers of graphs, 
especially for symmetric graphs 
such as complete multipartite graphs, Johnson graphs, Hamming graphs, 
etc (see \cite{KS}, \cite{KPS2010}, \cite{KPS2012}, \cite{KPS2011-}, 
\cite{PKS2009}, \cite{PS1}, \cite{PS2}, \cite{S}). 
}
\end{Rem}

\section*{Acknowledgment}

The author thanks the anonymous reviewers for helpful comments. 
The author was supported by JSPS Research Fellowships 
for Young Scientists. 


\end{document}